\documentclass[12pt,leqno]{amsart}%
\usepackage{times,amsmath,amsthm,amssymb,latexsym,amscd}
\usepackage[all]{xy}
\usepackage{graphicx}%
\usepackage{amsmath}%

\newtheorem{theorem}{Theorem}[section]

\newtheorem{prop}[theorem]{Proposition}
\newtheorem{cor}[theorem]{Corollary}

\theoremstyle{definition}
\newtheorem{definition}[theorem]{Definition}
\newtheorem{example}[theorem]{Example}

\newtheorem{nota}[theorem]{Notation}
\newtheorem{notar}[theorem]{Notation and Remarks}
\theoremstyle{remark}
\newtheorem{remark}[theorem]{Remark}
\newtheorem{remarks}[theorem]{Remarks}
\numberwithin{equation}{section}

\def\C{{\mathbb{C}}}
\def\R{{\mathbb{R}}}

\def\Z{{\mathbb{Z}}}

\def\G{{\Gamma}}
\def\go{{\G_1}}
\def\gt{{\G_2}}
\def\bs{{\backslash}}

\def\nab{{\nabla}}

\def\Lk{{L^{\otimes k}}}

\def\Spec{\operatorname{\textit{Spec}}}
\def\vol{\operatorname{vol}}

\begin{document}

\begin{title}[Quantum Equivalent Magnetic Fields]{Quantum Equivalent Magnetic
Fields that Are Not Classically
Equivalent\\
Champs magn\'etiques quantiquement \'equivalents mais classiquement
non-\'equivalents}\end{title}

\author{Carolyn Gordon}
\address{Department of Mathematics, Dartmouth College, Hanover, NH  03755}
\email{csgordon@dartmouth.edu}
\thanks{The first author and last authors were supported in part by NSF Grants
DMS 0605247 and DMS 0906169.   The third author was partially supported by
DFG Sonderforschungsbereich 647.}
\author{William D. Kirwin}
\address{CAMGSD, Departamento de Matematica,
Instituto Superior Tecnico,
Av. Rovisco Pais,
1049-001 Lisboa, PORTUGAL}
\email{will.kirwin@gmail.com}

\author{Dorothee Schueth}
\address{Institut f\"ur Mathematik, Humboldt-Universit\"at zu
Berlin, D-10099 Berlin, Germany}
\email{schueth@math.hu-berlin.de}

\author{David Webb}
\address{Department of Mathematics, Dartmouth College, Hanover, NH  03755}

\email{david.l.webb@dartmouth.edu}

%    General info

\dedicatory{Dedicated to Pierre B\'erard and Sylvestre Gallot on the occasion of
their  sixtieth birthdays.}

\begin{abstract}
We construct pairs of compact K\"ahler-Einstein manifolds
$(M_i,g_i,\omega_i)$ ($i=1,2)$ of complex dimension $n$ with the
following properties: The canonical line bundle $L_i=\bigwedge^n T^*M_i$
has Chern class $[\omega_i/2\pi]$, and for each integer $k$ the tensor powers
$L_1^{\otimes k}$ and $L_2^{\otimes k}$ are isospectral for
the bundle Laplacian associated with the canonical connection, while $M_1$
and $M_2$ -- and hence $T^*M_1$ and $T^*M_2$ -- are not homeomorphic.
In the context of geometric quantization, we interpret these examples as
magnetic fields which are quantum equivalent but not classically
equivalent. Moreover, we construct many examples of line bundles $L$,
pairs of potentials $Q_1$, $Q_2$ on the base manifold, and pairs of
connections $\nabla_1$, $\nabla_2$ on $L$ such that for each integer $k$
the associated Schr\"odinger operators on $L^{\otimes k}$ are isospectral.

\medskip
\noindent
R\'esum\'e:
On construit des couples de vari\'et\'es de K\"ahler-Einstein compactes
$(M_i,g_i,\omega_i)$ ($i=1,2$) de dimension complexe $n$
avec les propri\'et\'es suivantes: La premi\`ere classe de Chern associ\'ee
au fibr\'e en droites canonique $L_i=\bigwedge^n T^* M_i$ est
$\omega_i/2\pi$, et pour tout entier $k$, les puissances tensorielles
$L_1^{\otimes k}$ et $L_2^{\otimes k}$ sont isospectrales pour le
Laplacien associ\'e \`a la connexion canonique, mais $M_1$ et $M_2$
-- et, en cons\'equence, $T^*M_1$ et $T^*M_2$ -- ne sont pas
hom\'eomorphes. Dans le contexte de la quantification g\'eom\'etrique,
nous interpr\'etons ces examples comme des champs magn\'etiques qui sont
\'equivalents au sens quantique mais pas au sens classique.
En plus, on construit beaucoup d'exemples de fibr\'es en droites $L$,
de couples de potentiels $Q_1$, $Q_2$ sur la vari\'et\'e de base
et de couples de connexions $\nabla_1$, $\nabla_2$ telles que pour
tout entier $k$ les op\'erateurs de Schr\"odinger associ\'es
sur $L^{\otimes k}$ soient isospectraux.
\end{abstract}
\maketitle

\tableofcontents

\section{Introduction}

Let $L$ be a Hermitian line bundle over a closed Riemannian manifold $(M,g)$.
The Riemannian metric $g$ on $M$ and the connection $\nab$ on $L$
together give rise to a Laplace operator $\Delta$ acting on the space
$C^\infty(M,L)$ of smooth sections of $L$ by
\begin{equation}\label{BundleLaplacian}
\Delta=-\rm{trace}(\nab^2),
\end{equation}
where
\[
\begin{CD}
C^\infty(M,L)\xrightarrow{\nabla}  C^\infty(T^*M\otimes L)\xrightarrow{\nabla}
C^\infty(T^*M\otimes T^*M\otimes L)
\end{CD}
\]
are the connections on $L$ and on $T^*M\otimes L$ (the latter is obtained from
the Levi-Civita connection on $T^*M$ and the given
connection $\nab$ on $L$; we denote it by $\nab$ as well)
and the trace is with respect to the Riemannian metric $g$.
The connection $\nab$ gives rise to a connection,
and thus also a
Laplacian, on the $k$th tensor power $\Lk$ of $L$ over $M$ for each integer $k$.
We will denote its spectrum, which is necessarily discrete, by $Spec(L,\nab,
k)$.

How much information is
encoded in these spectra?  For example, do they determine the connection?  The
curvature of the connection?  The Chern class of the bundle?  The geometry of
the base manifold?  We will primarily focus on a variant of the second question.

A closed 2-form $\omega$ on a Riemannian manifold $(M,g)$ is sometimes viewed as
a magnetic field.   The classical Hamiltonian system for a
charged particle
moving in the magnetic field is given by $(T^{\ast}M, \Omega, H)$.  Here
$\Omega$ is the symplectic structure on the phase space $T^*M$ given by
$\Omega:=\omega_{0}+\pi^{\ast}\omega$, where $\omega_{0}$ is the Liouville form
and $\pi:T^{\ast}M\rightarrow M$ is the projection,
and the Hamiltonian $H$ is given by $H(q,\xi)=\frac{1}{2}g_q(\xi,\xi).$   If
$\frac{1}{2\pi}\omega$ represents an integer cohomology class, then there exists
a
complex line bundle $L$ with Chern class $[\frac{1}{2\pi}\omega]$.  Endow $L$
with a
Hermitian structure and a Hermitian connection with curvature $-i\omega$.
Through the procedure of geometric quantization, the space of square integrable
sections of $\Lk$ is viewed as the ``quantum Hilbert space,'' and the quantum
Hamiltonian is the operator
$\widehat{H}_{k}=-\frac{\hbar^{2}}{2}(-\Delta-\tfrac{1}{6}R)$ with
$\hbar=\frac{1}{k}$, where $R$ is the scalar curvature of $M$.   Thus we ask:

\begin{itemize}
\item Does the collection of all $Spec(L,\nab,k)$, $k\in\Z$, determine the
symplectic structure $\Omega$ on $T^*M$?  That is, does ``quantum equivalence''
of two
magnetic fields imply their ``classical equivalence''?
\end{itemize}

We answer this question negatively by example.
We consider the case in which $(M,g,\omega)$ is a K\"{a}hler manifold;
in fact, we focus on Hermitian locally symmetric spaces of noncompact type,
normalized such that the Einstein constant is~$-1$.
For such spaces, the line bundle with Chern
class~$[\omega/2\pi]$ is the canonical line bundle of $(M,g,\omega)$. We will
show that for every normalized, simply-connected irreducible Hermitian
symmetric space~$X$
of noncompact type of real dimension at least four, there exist arbitrarily
large finite families of Hermitian locally symmetric spaces
$(M_{i},g_{i},\omega_{i})$ covered by~$X$ such that
${\Spec}(L_{i},\nabla_{i},k)={\Spec}(L_{j},\nabla_{j},k)$ for all $k$ and all
$i,j$ (where $\nab_i$ is the canonical connection on the canonical line bundle)
but such that the cotangent bundles of the various $M_{i}$ are mutually
non-homeomorphic.
%Here $L_i$ is the canonical line bundle over $M_i$ and
%$\nabla_i$ is the connection arising from the Levi-Civita connection on $M$.
Hence, the phase spaces $(T^*M_j, \Omega_j)$ for the magnetic flows of the
various
$(M_{j},\omega_{j})$ are not symplectomorphic, and yet the measurable quantum
energy spectra are the same.  Our method is based on Sunada's isospectrality
technique along with D. B. McReynolds's recent construction of arbitrarily large
finite families of mutually isospectral locally symmetric spaces.

In the example outlined above, the classical phase spaces of the ``quantum
equivalent'' systems fail not only to be symplectomorphic, but even to be
homeomorphic.
In a companion article, we will
construct by a different method an example of quantum equivalent magnetic fields
on a fixed manifold $M$ (a torus) for which the associated symplectic structures
on $T^*M$ are not symplectomorphic.

Our technique is similar to that of R. Kuwabara \cite{Ku90}, who constructed
pairs of connections on a fixed line bundle $L$ over, for example, a Riemann
surface $M$ such that $Spec(L,\nab_1,k)=Spec(L,\nab_2,k)$ for all $k$.   In the
final section of this paper, we review and slightly extend his construction.

The paper is organized as follows:  In Section 2, we describe some of the
relevant framework of geometric
quantization, which will allow for a physical interpretation of the
isospectrality
results. This material is of course well-known to experts in geometric
quantization, but
we include it here in the hopes that it may be of interest to a wider audience.
In Section 3, we describe Sunada's technique in our context and show how it
leads to the examples described above of Hermitian locally symmetric spaces (of
real dimension four and higher) that are quantum equivalent  but not classically
equivalent.  We also address the case of Riemann surfaces.  Finally, in Section
4, we consider isospectral connections and potentials on a fixed line bundle.

This
article, like many others of the authors, was influenced by Pierre B\'erard's
work.
We are pleased to celebrate many years of friendship on the occasion of
his birthday.

\section{Geometric quantization}

\subsection{Hamiltonian system associated with a magnetic field}\label{class}\

On $\R^3$, a magnetic field may be viewed as an exact 2-form $\omega$,
identified with the curl of the magnetic potential field $A$.   The $1$-form
$\alpha=A^\flat$ defines a connection $\nab:=d-i\alpha$ on the
(trivial) Hermitian line bundle $\R^3\times \C$ with curvature
$-i\omega=-i\,d\alpha$.\footnote{The appearance of $i=\sqrt{-1}$ here is a
matter
of convention. We choose the convention which is common in mathematics,
specifically in geometric quantization.}

In analogy with the situation in $\R^3$, a closed 2-form $\omega$ on a
Riemannian manifold $(M,g)$ can be interpreted as a magnetic field. The
Hamiltonian system for a charged particle
moving in the magnetic field has phase space $(T^{\ast}M, \Omega)$ with
$\Omega:=\omega_{0}+\pi^{\ast}\omega$, where $\omega_{0}$ is the Liouville form
on $T^*M$ (that is, $\omega_0=-d\lambda$, where $\lambda$ is the canonical
$1$-form on the cotangent bundle), and $\pi:T^{\ast}M\rightarrow M$ is the
projection; see \cite{GSS}, for example.  The
classical trajectories of the particle are given by the Hamiltonian flow of the
(kinetic energy)
Hamiltonian $H(q,\xi):=\frac{1}{2}g_q(\xi,\xi)$.  When $\omega=0$, so that
$\Omega=\omega_0$, this flow is just the usual geodesic flow describing a free
particle moving on $M$.

\begin{nota}\label{ce}  We will say that $(M_1,g_1,\omega_1)$ and
$(M_2,g_2,\omega_2)$ are \emph{classically equivalent} if the associated
Hamiltonian systems $(T^{\ast}M_1, \Omega_1, H_1)$ and $(T^{\ast}M_2,
\Omega_2, H_2)$ are equivalent.
\end{nota}

\begin{notar}\label{prep} In case $[\omega/2\pi]$ is an integral cohomology
class, let $L$ be the line bundle over $M$ with Chern class
$[\omega/2\pi]$.  Endow $L$ with a Hermitian structure and let $P$ be the
associated principal circle bundle.  Let $\nabla^L$ be a
Hermitian connection on $L$ with curvature $-i\omega$.  The connection~$\nabla$
and the Riemannian metric on~$M$ give rise to a Riemannian metric $\widetilde{g}$ on
$P$, sometimes called a Kaluza--Klein metric. Consider the associated geodesic
flow
on
$T^{\ast}P$. The circle action on the principal bundle~$P$ gives rise to a
Hamiltonian action of the circle $S^{1}$ on $T^{\ast}P$. The symplectic
reduction of the geodesic flow on $P$ by the $S^{1}$ action yields the
Hamiltonian system of the magnetic flow on $T^{\ast}M$ described above. In this
brief description we have followed Kuwabara; see~\cite{Ku03} for more
information.

In preparation for Subsection~\ref{subsec:GQ}, we note that the connection
$\nab^L$ and the Riemannian metric $g$ on
$M$ give rise to a Laplace operator $\Delta$ on the space
$C^\infty(M,L)$ of smooth sections of $L$ given by (\ref{BundleLaplacian}). By
the usual construction, $\nab^L$
induces a Laplace operator, also denoted $\Delta$, on the space
${C^\infty}(M,{L^{\otimes k}})$,
where $L^{\otimes k}$ is the $k$th tensor power of $L$.  The space
${C^\infty}(M,L^{\otimes k})$ may be identified with the space
$C^\infty_k(P)$ of smooth complex-valued functions $f$ on $P$ satisfying the
equivariance condition
$f(\alpha.x)=\alpha^{-k}f(x)$ for $\alpha\in S^{1}$ and $x\in P$. The
Laplace operator on $C^\infty(M,{L^{\otimes k}})$ is unitarily equivalent to
the restriction of $\Delta_P-4k^2\pi^2$ to $C^\infty_k(P)$, where
$\Delta_P$ is the Laplace--Beltrami operator of $(P,\widetilde{g})$.

The space of smooth sections $C^\infty(M,L^{\otimes k})$
is endowed with the standard $L^2$ inner product given for smooth sections $s$
and $t$ by
\[\langle s,t\rangle := \left(\frac{k}{2\pi}\right)^n\int s.t\
\frac{\omega^n}{n!},\]
where $s.t$ denotes the pointwise Hermitian inner product on each fibre.
This inner product defines a Hilbert space consisting of square-integrable
sections of $L^{\otimes k}$, of which the space $C^\infty(M,L^{\otimes k})$ of
smooth sections is a dense subspace. The operator $\Delta$ is an unbounded
operator on this Hilbert space with dense domain $C^\infty(M,L^{\otimes k})$.
The theory of unbounded operators on Hilbert spaces is well-developed (see, for
example, the classic texts \cite[Vol II, Chap. 8]{RS} and \cite{BS}), and we
mention here only
that $\Delta$ admits a self-adjoint extension, still denoted by $\Delta$, with
dense domain $D$ containing the space of smooth sections of $L^{\otimes k}.$ In
the following, when we say that $\Delta$ is an operator on the Hilbert space of
$L^2$-sections, it is to be understood in this usual sense of an unbounded
operator with dense domain.
\end{notar}

\subsection{Quantization of the Hamiltonian system}\label{subsec:GQ}\

Using geometric quantization, one associates to a classical mechanical system
(satisfying suitable requirements) a quantum mechanical system, consisting of a
Hilbert space $\mathcal{H}_{k}$ and a quantum Hamiltonian operator
$\widehat{H}_{k}:\mathcal{H}_{k}\to \mathcal{H}_{k}$ for each $k\in\mathbb{Z}$.
(Here Planck's constant is given by $\hbar=1/k$.)   For the Hamiltonian system
$(T^*M, \Omega,
H)$ in Subsection~\ref{class},  the quantization may be carried out provided
that
$\omega/2\pi$ represents an integral cohomology class of $M$.   In this case,
one obtains Theorem~\ref{hk} below. Following the statement of the theorem
and related remarks, we will briefly outline the procedure of geometric
quantization. For a complete presentation, see the classic references
\cite{Woo} and~\cite{Sni}; see also \cite{GS} and \cite{BW}.

\begin{theorem}\label{hk}
\label{thm:Q(H)}\cite[p. 204]{Woo}
%, \cite[eqn (10.59)]{Sni}
We use
Notation~\ref{prep} and assume that $[\omega/2\pi]$ is an integral cohomology
class. The quantum Hilbert space associated to the classical Hamiltonian
system
$(T^{\ast}M,\Omega,H)$ of Subsection~\ref{class} is given for each integer $k$
by
$\mathcal{H}_{k}=L^{2}(M,L^{\otimes k})$ (the space of square-integrable
sections of $\Lk$) and the quantization of the Hamiltonian $H$ is the
(unbounded) operator
\begin{equation}\label{Hamiltonian}
\widehat{H}_{k}=-\frac{\hbar^{2}}{2}(-\Delta-\tfrac{1}{6}R)
\end{equation}
on $\mathcal{H}_{k}$, where $R$ is the scalar curvature of the metric $g$.
(Here $\hbar=
\frac{1}{k})$.
\end{theorem}

The allowed energy values of the charged particle in the magnetic field, which
are what one would see if one
``measured'' the energy of the quantum particle, are the eigenvalues of
$\widehat{H}_{k}$.

\begin{remark}\label{amb} The definition of the Laplacian $\Delta$ on $\Lk$,
and thus of the operators $\widehat{H}_{k}$, depends on a choice of connection on
$L$
with the specified curvature~$-i\omega$.   However, in the examples that we will
give in
Subsection~\ref{ex}, there will be a natural choice of connection
with that curvature.
\end{remark}

\begin{nota}\label{qe} Let ($M_i, g_i)$, $i=1,2$, be a compact Riemannian
manifold and
let $\omega_i$ be a closed $2$-form on $M_i$ such that $[\omega/2\pi]$
is an integral cohomology
class.   For each integer $k$, let $\widehat{H}_{k}^i:L^{2}(M,L_i^{\otimes
k})\rightarrow
L^{2}(M,L_i^{\otimes k})$ be the associated quantum Hamiltonian as given in
Theorem~\ref{hk}.
  We will say that $(M_1,
\omega_1)$ and $(M_2,\omega_2)$ are \emph{quantum equivalent} (with respect to
the connections used to define the line bundle Laplacians) if for every $k$, the
operators $\widehat{H}_{k}^1$ and $\widehat{H}_{k}^2$ have the same spectrum.
\end{nota}

We now outline the quantization procedure.  Consider the classical Hamiltonian
system $(T^*M,\Omega, H)$ given in Subsection~\ref{class}.  Recall that
$\Omega=\omega_0+\pi^*\omega$.  Let $\pi^{\ast}L$ be the pullback of
$L$  to a bundle over $T^*M$ and $\pi^*\nab^L$ the pullback of the connection.
Since the Liouville form $\omega_0$ on $T^*M$ is exact, the Hermitian line
bundle $L_\Omega$ with Chern class $[\Omega/2\pi]$ may be identified
with $\pi^{\ast}L$.  Writing $\omega_0=d\Theta$, the Hermitian connection
$\nab:=\pi^{\ast}\nab^L -i\Theta$ on $L_\Omega$ has curvature $-i\Omega$.

The \emph{prequantization} of the Hamiltonian system $(T^{\ast}M,\Omega)$ is the
space of square-integrable sections of
$L_{\Omega}^{\otimes k}$
%:=(L_\Omega)^{otimes k}$
with respect to the standard inner product
\[
\langle s,t\rangle:=\left(  \frac{k}{2\pi}\right)
^{n}\int_{T^{\ast}M}s.t~\frac{\Omega^{n}}{n!},
\]
where $s.t$ denotes the (pointwise) Hermitian product on $L_{\Omega}^{\otimes
k}.$  For a smooth function $f$ on $T^*M$ (we are interested in particular in
the Hamiltonian $H$ above), one associates a \emph{prequantum Hamiltonian
operator} $\widehat{f}^{preQ}$, given by the Kostant--Souriau construction:
\begin{equation}
\widehat{f}^{preQ}:=\frac{i}{k}\nabla_{X_{f}}^{L_{\Omega}^{\otimes k}
}+f,\label{eqn:K-Squantization}
\end{equation}
where $X_{f}$ is the Hamiltonian vector field associated to $f$, defined by
$\Omega(X_{f},\cdot)=df(\cdot)$. The Kostant--Souriau prequantization
(\ref{eqn:K-Squantization}) satisfies Dirac's quantization conditions:
\begin{enumerate}
\item the map $f\mapsto\widehat{f}$ $^{preQ}$is linear,
\item the quantization $\widehat{1}^{preQ}$ of the constant map $1$ is the identity
operator, and
\item $[\widehat{f}^{preQ},\widehat{g}^{preQ}]=-i\hbar\widehat{\{f,g\}}^{preQ},$ where
$\{\cdot,\cdot\}$ is the Poisson bracket, $[\cdot,\cdot]$ is the operator
commutator, and $\hbar=1/k$.
\end{enumerate}
Indeed, the prequantization (\ref{eqn:K-Squantization}) is derived precisely so
that it satisfies (1) -- (3) above. (See \cite[Chap. 8]{Woo}). Unfortunately,
the
pair $(L^{2}(T^{\ast}M,L_{\Omega}^{\otimes k}),\,f\mapsto\widehat{f}^{preQ})$ does
not define a ``good'' quantization, essentially because
$L^{2}(T^{\ast}M,L_{\Omega}^{\otimes k})$ is too big. For example, in the case
$\omega=0$ and $M=\mathbb{R}^{n}$, which corresponds to a free particle moving
in Euclidean space, the line bundle $L_{\Omega}^{\otimes k}$ is trivial and the
prequantum Hilbert space is then
$L^{2}(T^{\ast}\mathbb{R}^{n}=\mathbb{R}^{n}\times\mathbb{R}^{n}).$ The
variables in the first $\mathbb{R}^{n}$ factor give the position of the
particle, and the variables in the second $\mathbb{R}^{n}$ factor describe the
momentum. But one knows from quantum mechanics that a wave function cannot be
simultaneously a function of both position \emph{and} momentum.

In order to obtain a Hilbert space of the ``correct'' size, one first
chooses a polarization of $(T^*M,\Omega)$. A \emph{polarization} of a
symplectic manifold is an integrable (real or complex) Lagrangian distribution.
If, as is the case here, the phase space is a cotangent bundle, one may take the
\emph{vertical
polarization}, i.e., the distribution given by the tangent spaces to the fibers
of $T^{\ast}M$. (Note that this distribution is indeed Lagrangian with respect
to $\Omega$ as well as $\omega_0$.)  This means that we are considering wave
functions which depend only
on position, not on momentum. Thus, the vertical polarization corresponds to
the ``position-space representation'' in quantum mechanics.

Once a choice of
polarization $\mathcal{P}$ is made, one would ideally like to define the quantum
Hilbert space to be
the subspace $L^2_\mathcal{P}(T^*M,L_{\Omega}^{\otimes k})$ of the prequantum
Hilbert space $L^2(T^*M,L_{\Omega}^{\otimes k})$ consisting of those sections
that are covariantly constant in the $\mathcal{P}$ directions and then restrict
the Kostant-Souriau prequantum Hamiltonians $\widehat{f}^{preQ}$ to this subspace.
However, there are two problems here.  First, $L^2_\mathcal{P}=\{0\}$!
Secondly, even for
a polarization $\mathcal{P}$ that yields a nontrivial quantum Hilbert
space,\footnote{A
typical example of such a polarization is available whenever $T^*M$ admits a
K\"ahler structure, for example when $M$ is a compact Lie group. In this case,
one can take $\mathcal{P}$ to be the holomorphic tangent bundle.} the
Kostant-Souriau
quantization $\widehat{f}^{preQ}$ does not in
general preserve the quantum Hilbert space. Indeed, $\widehat{f}^{preQ}$ will only
preserve $\mathcal{H}_{k}$ if the Hamiltonian flow of $f$ preserves the
polarization $\mathcal{P}$. One can show that the $\Omega$-Hamiltonian flow of
$H$ does
not preserve \emph{any} polarization.

Fortunately, there is only one
more piece of the puzzle which will remedy both of the remaining problems at
once: the so-called half-form correction. The half-form correction boils down to
tensoring $L^{\otimes k}$ with a square root of the canonical bundle associated
to the polarization. Sections of such a bundle are called half-forms.

The half-form correction, due essentially to Blattner, Kostant and Sternberg,
allows one
to quantize a larger set of functions than just those whose flows preserve the
polarization, and in particular one can quantize the standard Hamiltonians which
appear in wave mechanics, of which our $H=\frac{1}{2}\Vert\xi\Vert^{2}$ is an
example. Moreover, the quantum Hilbert space associated to the vertical
polarization, in the presence of the half-form correction, will turn out to be
just $L^2(M,L^{\otimes k})$, which is exactly what one would naively expect for
the position-space representation.\footnote{There are several further
advantages, from both
the mathematical and physical viewpoints, though they are not relevant to our
current purposes. One, which is easy to describe, is that when using the BKS
construction to quantize the simple harmonic oscillator (a well-known example
from physics), a shift is introduced which results in the physically correct
energy spectrum. Specifically, without the BKS construction, one obtains an
energy spectrum consisting of integer multiples of $\hbar.$ The physically
correct spectrum, which is obtained using the BKS construction, is
$\{(n+\frac{1}{2})\hbar:n\in\mathbb{Z\}}$.}

The BKS construction in our setting is as follows. (We refer the interested
reader to \cite{Woo}, Chap. 9, for more details and proofs; see also \cite{GS}.)
Choose a line
bundle $\delta$ such that $\delta\otimes\delta=\bigwedge^{n}TM$ (this is
possible because $\bigwedge^{n}TM$ is trivializable), and let $\nu$ be a section
of $\delta$ with $\nu^{2}=\vol_g(M),$ where $\vol_g(M)$ is the Riemannian volume
form on
$M$. Sections of $\delta$ are called half-forms (associated to the vertical
polarization), and the half-form corrected
quantum Hilbert space is defined to be
\[
\widehat{H}_k:=L^{2}_{\mathcal{P}}(T^*M,L_\Omega^{\otimes k}\otimes\pi^*\delta)
\]
where the inner product is defined by the canonical pairing of half-forms. In
particular, a section of $L_\Omega\rightarrow T^*M$ which is vertically
covariantly constant is uniquely determined by its value on the zero-section
$M$, and the inner product of two such sections is therefore given
by\footnote{We will abuse notation slightly and not distinguish between $\nu$
(or
$\vol_g(M)$) and its pullback $\pi^{\ast}\nu$ (resp. $\pi^{\ast}\vol_g(M)$).}
\begin{equation}
\langle s\nu,t\nu\rangle = \left(\frac{k}{2\pi}\right)^n\int_M s.t\ \vol_g(M).
\label{eqn:half-form-pairing}
\end{equation}
Hence, we see that the quantum Hilbert space associated to the vertical
polarization can be identified with $L^2(M,L^{\otimes k}).$

\medskip

Now that we have the correct quantum Hilbert space, we need to quantize the
Hamiltonian flow of the kinetic energy $H$. Let $\rho_{t}$ denote the
Hamiltonian flow of $H$ on $T^{\ast}M$. In order to
define the quantization of the Hamiltonian $H$, we evolve $\psi\nu$ for a short
time (that is, apply $\exp(-ikt\widehat{H}^{preQ})$ to the first factor, and the
pull-back $\rho_{t}^{\ast}$ to the second factor), and then project the result
back into $\widehat{\mathcal{H}}_{k}.$

The projection is achieved by a generalization of the half-form pairing
(\ref{eqn:half-form-pairing}). One can show that the pushforward of the vertical
polarization by $\rho_{t}$ is an integrable
Lagrangian distribution which is (at least for small $t$) transverse to the
vertical polarization. Hence, there exists some function $f_{t}\in
C^{\infty}(T^{\ast}M)$ such that
$\rho_{t}^{\ast}(\vol_g(M))\wedge\vol_g(M)=f_{t}\Omega^{n}/n!$. The generalized
(BKS) half-form pairing is
then defined to be
\[
\left(  \rho_{t}^{\ast}\nu\right)  .\nu:=\sqrt{f_{t}}.
\]

This pairing can be shown to be nondegenerate (at least for small $t$), and
therefore defines a bijection between
$(\exp(-ikt\,\widehat{H}^{preQ})\otimes\rho_{t}^{\ast})\widehat{\mathcal{H}}_{k}$
and $\widehat{\mathcal{H}}_{k}.$ The quantum Hamiltonian $\widehat{Q}_{k}(H)$ is
obtained by computing the derivative with respect to $t$, evaluated at $t=0$, of
the operator on $\widehat{\mathcal{H}}_{k}$ given by first applying
$(\exp(-ikt\widehat{H}^{preQ})\otimes\rho_{t}^{\ast})$ and then projecting the
result back into $\widehat{\mathcal{H}}_{k}$ using the BKS pairing. At the end
of the day, we are really interested only in sections of $L^{\otimes k};$ thus,
we can take a section $\psi$ in $\mathcal{H}_{k},$ multiply it by $\nu$, apply
the BKS construction, and write the result in the form $\psi^{\prime}\nu.$ The
quantization of the Hamiltonian $H$ is then defined to be
$\widehat{H}_{k}\psi:=\psi^{\prime}$. In our case, this yields the expression in
Theorem~\ref{hk}.

\section{You can't hear a magnetic field}

\subsection{The Sunada technique}\label{st}

We will use a variant of Sunada's technique \cite{Su}.

\begin{definition}
\label{def.ac} Let $G$ be a finite group and let ${{\Gamma}_{1}}$ and
${{\Gamma}_{2}}$ be subgroups of $G$. We will say that ${{\Gamma}_{1}}$ is
\emph{almost conjugate} to ${{\Gamma}_{2}}$ in $G$ if
there is a bijection $\Gamma_1\to\Gamma_2$ carrying each element of $\Gamma_1$
to a conjugate element in $\Gamma_2$; equivalently,
each $G$-conjugacy class
$[g]_{G}$ intersects ${{\Gamma}_{1}}$ and ${{\Gamma}_{2}}$ in the same number of
elements.
\end{definition}

Sunada's Theorem states that if a finite group $G$ acts by isometries on a
compact Riemannian manifold $M$ and if ${{\Gamma}_{1}}$ and ${{\Gamma}_{2}}$ are
almost conjugate subgroups of $G$ acting freely on $M$, then
${{\Gamma}_{1}}\backslash M$ and ${{\Gamma}_{2}}\backslash M$ are isospectral.

\begin{remarks}\label{rem.repcond}$\text{ }$

\begin{enumerate}
\item The almost conjugacy condition is equivalent to a
representation theoretic condition as follows. The right multiplication of $G$
on the cosets in ${\Gamma}_{i}{\backslash} G$ gives rise to a natural action of
$G$ on the finite-dimensional vector space
${\mathbf{R}}[{\Gamma}_{i}{\backslash} G]$. The subgroups ${{\Gamma}_{1}}$ and
${{\Gamma}_{2}}$ of $G$ are almost conjugate if and only if there exists an
isomorphism
\[
\tau: {\mathbf{R}}[{{\Gamma}_{1}}{\backslash} G]\to{\mathbf{R}}[{{\Gamma}_{2}
}{\backslash} G]
\]
intertwining the actions of $G$.

\item Assume that ${{\Gamma}_{1}}$ and ${{\Gamma}_{2}}$ are almost conjugate in
$G$ and let $\tau$ be the intertwining map in (i). Let $W$ be any vector space
on which $G$ acts on the right. For $i=1,2$, let $W^{{\Gamma}_{i}}$ be the
subspace of vectors fixed by all elements of ${\Gamma}_{i}$. Then $\tau$ gives
rise to a linear isomorphism, called ``transplantation''
\[
\mathcal{T}:W^{{{\Gamma}_{2}}}\to W^{{{\Gamma}_{1}}}.
\]
Transplantation was first introduced in an example in \cite{Bu86} and
systematized in \cite{Be92} to give a new proof of Sunada's Theorem; see also
\cite{Z}. We are following the presentation in \cite{GMW}.

\item Transplantation is functorial: if $V$ and $W$ are right
$G$-spaces and $\psi:V\to W$ is a $G$-equivariant map, then the following
diagram commutes:
\[
\begin{CD}
W^{\gt} @>{\mathcal{T}_W}>> W^{\go} \\
@V{\psi}VV                 @VV{\psi}V \\
V^{\gt} @>{\mathcal{T}_V}>> V^{\go}\\
\end{CD}
\]
Moreover, if $W$ is an inner product space and if the $G$ action is unitary,
then the transplantation map is unitary.
\end{enumerate}
\end{remarks}

\begin{nota}
Given a Hermitian line bundle $L$ over a closed Riemannian manifold~$(M,g)$
and a Hermitian connection~$\nabla$ on~$L$ we denote by $\Spec(L,\nabla,k)$
the spectrum of the associated Laplace operator~$\Delta$
on $C^\infty(M, L^{\otimes k})$ (recall Notation and Remarks~\ref{prep}).
For a potential $Q\in C^\infty(M)$, we denote by $\Spec(Q;L,\nabla,k)$
the spectrum of $\Delta+Q$ on $C^\infty(M, L^{\otimes k})$.
\end{nota}

\begin{prop}
\label{mt}  Let $(M,g)$ be a compact Riemannian manifold, let $L$ be a Hermitian
line bundle over $M$, and let ${\nabla}$ be a Hermitian connection on $L$. Let
$G$ be a finite group that acts on $L$ carrying fibers to fibers, preserving
$\nabla$, and such that the induced action on $M$ is by isometries. For $i=1,2$,
suppose that ${\Gamma}_{i}$ is a subgroup of $G$ whose action on $M$ is free.
Thus $L_{i}:={\Gamma}_{i}{\backslash} L$, $i=1,2$ is a Hermitian line bundle
over $M_{i} :={{\Gamma}_{i}{\backslash} M}$, and $\nabla$ induces a connection
${\nabla}_{i}$ on $L_i$. If\/ ${\Gamma}_{1}$ and ${\Gamma}_{2}$ are almost
conjugate in $G$, then{\rm:}
\begin{itemize}
\item[(i)]
\[
{\Spec}(L_{1},\nabla_{1},k)={\Spec}(L_{2},\nabla_{2},k)
\]
for all positive integers $k$.
%$$\xymatrix{
%{} & M \ar @{->>}[dl] \ar @{->>}[dr] & \\
%M_1:=\G_1\bs M\ar   & & \G_2\bs M=:M_2\\
%}$$

\item[(ii)]
If, moreover, $Q\in C^{\infty}(M)$ is a $G$-invariant function,  then
\[
{\Spec}(Q; L_{1},\nabla_{1},k)={\Spec}(Q; L_{2},\nabla_{2},k)
\]
for all positive integers $k$, where we use the same notation $Q$ for the smooth
potentials on $M_{1}$ and $M_{2}$ induced by the potential $Q$ on $M$.
\end{itemize}
\end{prop}

This variant of Sunada's Theorem is essentially contained in R. Kuwabara
\cite{Ku90}, although his interest was in pairs of connections on the same
underlying  bundle and in the case $Q=0$.

For a proof by transplantation, observe that $G$ acts on the right on the space
${C^\infty}(M,{L^{\otimes k}})$ of smooth sections of ${L^{\otimes k}}$ by
$(f.g)(x)=g^{-1}.f(g.x)$ for $f\in{C^\infty}(M,{L^{\otimes k}})$, $g\in G$, and
$x\in M$. The space ${C^\infty}(M_i,{L_{i}^{\otimes k}})$ of smooth sections of
${L_{i}^{\otimes k}}$ may be identified with the space
${C^\infty}(M,{L^{\otimes k}})^{{\Gamma}_{i}}$ of ${\Gamma}_{i}$-invariant
elements of ${C^\infty}(M,{L^{\otimes k}})$. Thus by Remark \ref{rem.repcond},
we obtain a transplantation map $\mathcal{T}:{C^\infty}(M_2,{L_{2}^{\otimes
k}})\rightarrow{C^\infty}(M_1,{L_{1}^{\otimes k}})$.
%Moreover, with this identification, the Laplacian
%$\Delta_i$ on $\E(\Lik)$ (associated with the Riemannian metric on $M_i$ and
%the
%connection $\nab_i$)  is the restriction to $\E(\Lk)^{\G_i}$ of the Laplacian
%$\Delta$ of $\Lk$.  Since $\Delta$ commutes with the action of $G$, we may let
%$\Delta$
%play the role of $\psi$ in Remark \ref{rem.repcond}.  It follows that
%$\mathcal{T}$ intertwines the Laplacians on $\Lok$ and $\Ltk$, thus proving the
%theorem.
Moreover, with this identification, the Schr\"{o}dinger operator $\Delta_{i}+Q$
on ${C^\infty}(M_i,{L_{i}^{\otimes k}})$ (associated with the Riemannian metric
on $M_{i}$, the connection ${\nabla}_{i}$, and the potential~$Q$) is the
restriction to ${C^\infty}(M,{L^{\otimes k}})^{{\Gamma}_{i}}$ of the
Schr\"{o}dinger operator $\Delta+Q$ of ${L^{\otimes k}}$. Since $\Delta$
commutes with the action of $G$ and since $Q$ is $G$-invariant, we may let
$\Delta+Q$ play the role of $\psi$ in Remark~\ref{rem.repcond}. It follows that
$\mathcal{T}$ intertwines the Schr\"{o}dinger operators $\Delta_{1}+Q$ and
$\Delta_{2}+Q$ on ${L_{1}^{\otimes k}}$ and ${L_{2}^{\otimes k}}$, thus proving
the theorem.

\begin{theorem}
\label{mc}
We use the notation and hypotheses of Proposition~\ref{mt}, part~(i).  Let
$-i\omega_j$ be the curvature of the connection $\nab_j$ on $L_j$, $j=1,2$.
Then
in the language of Notation~\ref{qe} and Remark~\ref{amb}, $(M,\omega_1)$ and
$(M_2,\omega_2)$ are quantum equivalent with respect to the connections
$\nab_1$ and $\nab_2$.
\end{theorem}

\begin{proof}  We apply part~(ii) of Proposition~\ref{mt} with the scalar
curvature $R$ of $M$ in the role of $Q$.  Note that $R$ is necessarily
$G$-invariant since $G$ acts by isometries on $M$.
\end{proof}

\subsection{Construction of examples}\label{ex}\

Let $(M, g, \omega)$ be a K\"ahler manifold of complex dimension $n$. (Here
$\omega$ is the K\"ahler
form.) The canonical line bundle $L_{M}$ over $M$ is defined to be the $n$th
exterior power of the holomorphic cotangent bundle. Since $M$ is K\"ahler, the
Levi-Civita connection on $TM$ commutes with the complex structure and thus
defines a holomorphic connection on the holomorphic tangent bundle. This
connection gives rise to a holomorphic connection on $L_{M}$ that we will call
the \emph{canonical connection}.

If $X$ is a simply-connected Hermitian
symmetric space of non-compact type and $M$ is
a compact locally symmetric space with universal covering $X$, we will call $M$
an $X$-space. Every $X$-space $M$ is a Hodge manifold, i.e., $M$ is a K\"ahler
manifold and a suitable real multiple
of the K\"ahler form $\omega$ of $M$ represents an integer cohomology
class. More precisely, if the metric is rescaled such that $X$ (and hence
each $X$-space $M$) has Einstein constant~$-1$
then the Chern class of the canonical bundle $L_{M}$ is
$[\omega/2\pi]$ (see~\cite{Ba}, formulas (4.68) and~(4.59);
compare also~\cite{We}, p. 219.) As in Remark~\ref{amb}, the notion of
``quantum equivalence'' of $(M_1,\omega_1)$ and $(M_2,\omega_2)$, where the
$M_i$ are $X$-spaces and $\omega_i$ their K\"ahler forms, will mean with respect
to the canonical connections on the canonical bundles $L_{M_i}$.

%(iii) In the setting of (ii), we will denote by $\Spec(M,\omega)$ the
%collection
%of all $\Spec(L_M,\nab,k)$ (more precisely, the function associating
%$\Spec(L_M,\nab,k)$ with~$k$), where $\nab$ is the canonical connection on the
%canonical line bundle $L_M$.

\begin{theorem}\label{herm}
Let $X$ be a simply-connected Hermitian symmetric space of noncompact type of
real dimension at least four. Then there exist arbitrarily large families of
non-isometric $X$-spaces $M_{i}$ such that the $(M_i,\omega_i)$ are all mutually
quantum equivalent but not classically equivalent. In fact the phase spaces
$(T^*M_i, \Omega_i)$ of the classical Hamiltonian systems are not
symplectomorphic {\rm(}or even homeomorphic{\rm)}.
\end{theorem}

\begin{proof}
D.B. McReynolds \cite{Mcr} showed, using Sunada's Theorem, that for every
simp\-ly-con\-nected symmetric space $X$ of non-compact type, there exist
arbitrarily large collections of non-isometric $X$-spaces $M_{i}$ whose
Laplace-Beltrami operators are mutually isospectral. For each such collection,
there exists an $X$-space $M$ and a finite group $G$ of isometries of $M$ such
that $M_{i}={{\Gamma}_{i}{\backslash}M}$, where the ${\Gamma}_{i}$ are almost
conjugate subgroups of $G$. In the setting that $X$ is Hermitian symmetric, the
isometries are holomorphic. Since all holomorphic isometries of $M$ preserve
both the canonical bundle and the canonical connection, we can now apply
Theorem~\ref{mc} to see that $(M_{i},\omega_{i})$ and $(M_{j},\omega_{j})$
are quantum equivalent for all $i,j$.

Mostow Strong Rigidity tells us that the various $M_{i}$ have non-isomorphic
fundamental groups, i.e., $M_{i}=\Lambda_{i}{\backslash}X$ with $\Lambda_{i}$
and $\Lambda_{j}$ non-isomorphic discrete uniform subgroups of the group of
isometries of $X$ when $i\neq j$. The cotangent bundle $T^{\ast}M_{i}$ is the
quotient of the (trivial) bundle $T^{\ast}X$ by the action of $\Lambda_{i}$ and
thus the various $T^{\ast}M_{i}$ are also non-homeomorphic.
\end{proof}

The assumption on the dimension of $X$ in Theorem~\ref{herm}, equivalently the
exclusion of the case that $X$ is the real hyperbolic plane, was needed only so
that the phase spaces for the classical Hamiltonian systems would not be
homeomorphic.  In fact, we have the following:

\begin{prop} Let $(M_1,g_1)$ and $(M_2,g_2)$ be \emph{any} pair of
hyperbolic Riemann surfaces which are iso\-spec\-tral with respect to the
Laplace-Beltrami operator acting on functions.
Then $(M_1,\omega_1)$ and $(M_2,\omega_2)$ are quantum
equivalent, where $\omega_i$ is the K\"ahler form of
$(M_i,g_i)$.
\end{prop}

We emphasize that, in contrast to the Hermitian locally symmetric spaces in
Theorem~\ref{herm}, the Riemann surfaces are \emph{not} required to satisfy the
conditions of Sunada's Theorem.

\begin{proof}
H. Pesce \cite{Pes} proved that every pair of isospectral compact
Riemann surfaces $(M_1,g_1)$ and $(M_2,g_2)$ is \emph{strongly} isospectral in
the following sense:  Let $G=PSL(2,\R)$.  A Hermitian vector bundle $E$ over the
hyberbolic plane $X$ is said to be homogeneous if $G$ acts on $E$, carrying
fibers to fibers, such that the induced action on $X$ is the standard action by
isometries.  The actions of $G$ on $M$ and $E$ give rise to an action of $G$ on
the space of smooth sections of $E$.  A self-adjoint elliptic differential
operator $D$ on $E$ (i.e., on smooth sections of $E$) is said to be
\emph{natural} if it commutes with the $G$-action.  In that case, if $\G$ is a
discrete subgroup of $G$ acting freely and properly discontinuously on $X$, then
$D$ induces a self-adjoint elliptic differential operator on the bundle $\G\bs
E$ over the Riemann surface $\G\bs M$.  Compact Riemann surfaces $M_1=\G_1\bs X$
and $M_2=\G_2\bs X$ are said to be \emph{strongly isospectral} if for each
homogeneous Hermitian vector bundle $E$ over $X$ and each natural self-adjoint
elliptic operator $D$ on $E$, the induced operators on the bundles $\G_1\bs E$
over $M_1$ and $\G_2\bs E$ over $M_2$ are isospectral.   (Aside:  The key point
in proving that isospectral compact Riemann surfaces $M_1=\G_1\bs X$ and
$M_2=\G_2\bs X$ are always strongly isospectral is that isospectrality of the
Riemann surfaces implies that the representations of $G$ induced by the trivial
representations of $\G_1$ and $\G_2$ are equivalent.  This condition is
considerably weaker than the Sunada condition, which requires that $\G_1$ and
$\G_2$ be subgroups of some finite subgroup $\G$ of $G$ and that the trivial
representations of $\G_1$ and $\G_2$ induce equivalent representations of $\G$.)

The proposition follows from that fact that the canonical Hermitian line bundle
over $X$ and all its tensor powers are homogeneous, and the Laplacian associated
with the canonical connection is natural.
\end{proof}

\begin{remark}  Using Sunada's technique, R. Brooks, R. Gornet, and W. Gustafson
\cite{BGG} constructed arbitrarily large finite families of mutually
isospectral, non-isometric Riemann surfaces.  (Their work motivated that of
D.B.~McReynolds cited above.)  While the vast majority of known isospectral
Riemann surfaces were constructed by Sunada's technique, M.-F.~Vign\'eras's
examples  \cite{V} and recent examples of  C.S.~Rajan \cite{R} do not satisfy
the Sunada condition.
\end{remark}

\section{Isospectral connections and potentials on a line bundle and its
tensor powers}

Using a trick introduced by R. Brooks \cite{Br}, we can use Proposition~\ref{mt}
to obtain isospectral connections and potentials on a single line bundle and its
tensor powers.

\begin{cor}
\label{corm} In addition to the hypotheses of Proposition~\ref{mt}, assume that
there exists a bundle map $\sigma$ of $L$, projecting to an isometry (also to be
denoted $\sigma$) of $M$, such that $\sigma$ normalizes $G$ and such that
$\sigma{{\Gamma}_{1}}\sigma^{-1}={{\Gamma}_{2}}$. Continue to denote by $\sigma$
the induced bundle map from ${L_{1}^{\otimes k}}$ to ${L_{2}^{\otimes k}}$. Then
\[
{\Spec}(Q; L_{1},\nabla_{1},k)={\Spec}(\sigma^{*}Q;
L_{1},\sigma^{*}\nabla_{2},k)
\]
for all positive integers $k$.
\end{cor}

This corollary is contained in Kuwabara~\cite{Ku90} for the case $Q=0$.
\begin{remark} One may choose $\nab$ to be $\sigma$-invariant as well as
$G$-invariant, in which case $\nab_1=\sigma^*\nab_2$.   We can then conclude
that $Q_1$ and $\sigma^*Q_2$ are isospectral potentials for the Schr\"odinger
operator $-\Delta_1 + $potential.
\end{remark}

We now explain how to use the corollary to obtain examples in which the base
manifolds are Riemann surfaces. Brooks ~\cite{Br} gave explicit examples of
finite groups~$G$ and Riemann surfaces $(M,g)$ (with a hyperbolic Riemannian
metric~$g$) such that the following conditions are satisfied:

\begin{itemize}
\item[(i)] The group $G$ acts freely by orientation preserving isometries on the
oriented Riemann surface $(M,g)$.

\item[(ii)] There exists a pair of almost conjugate, nonconjugate subgroups
$\Gamma_{1}$, $\Gamma_{2}$ of~$G$.

\item[(iii)] There exists an outer automorphism~$\tau$ of~$G$ such that
$\Gamma_{2} =\tau\Gamma_{1}\tau^{-1}$ and such that the action of~$G$ extends to
a free action of the semi-direct product $\widehat G$ of $G$ and $\langle
\tau\rangle$ on~$(M,g)$ by orientation-preserving isometries.
\end{itemize}

Using these objects we obtain the following class of examples.

\begin{example}
\label{surf} We choose $(M,g)$, $G$, $\Gamma_{1}$, $\Gamma_{2}$, $\tau$,
$\widehat G$ as above and consider the Hermitian line bundle $L_{N}$ over $N:=\widehat
G\backslash M$. Denote its pullback to~$M$ by~$L$. The group~$\widehat G$ acts
on~$L$ by vector bundle isomorphisms. We choose a $\widehat G$-invariant Hermitian
connection $\widehat{\nabla}$ on~$L$ by pulling back a Hermitian connection from
$L_{N}$, and we choose a function $f\in C^{\infty}(M)$ which is $G$-invariant
but not $\tau$-invariant. Denoting the Riemannian volume form on~$M$
by~$\omega$, we let ${\nabla}:=\widehat{\nabla}+i\,d^*(f\omega)$. Note that
${\nabla}$ is $G$-invariant, but not $\tau$-invariant. Moreover, we choose any
$G$-invariant potential $Q\in C^{\infty}(M)$. Finally, we let~$\sigma$ denote
the vector bundle isomorphism of~$L$ induced by~$\tau$. Applying
Proposition~\ref{mt} together with Corollary~\ref{corm} we obtain, for the
vector bundle $L_{1}:=\Gamma_{1}\backslash L$ over $M_{1}:=\Gamma_{1}\backslash
M$ and the induced connections ${\nabla}_{1}$ on~$L_{1}$, resp.~${\nabla}_{2}$
on $L_{2}:=\Gamma_{2}\backslash L$:
\[
{\Spec}(Q;L_{1},{\nabla}_{1},k)={\Spec}(\sigma^{*}Q;L_{1},\sigma^{*}{\nabla}_{2}
,k)
\]
for all $k$.
\end{example}

\begin{remark}

(i) The choice of~${\nabla}$ in the previous example guarantees that the
resulting
pairs of isospectral connections ${\nabla}_{1}$ and~$\sigma^{*}{\nabla}_{2}$
have different curvature. In fact, the pullbacks to~$L$ of the connections
${\nabla}_{1}$ and $\sigma^{*}{\nabla}_{2}$ on~$L_{1}$ are $\widehat{\nabla}+
i\,d^*(f\omega)$ and $\widehat{\nabla}+ i\,d^*((\tau^{*}f)\omega)$,
respectively. The pullback to~$M$ of the difference of the corresponding
curvature forms on~$M_{1}$ is given by $i\,d d^*((f-\tau^{*}f)\omega)$.
The $2$-form $(f-\tau^{*}f)\omega$ has integral zero over~$M$ and is thus exact
by Poincar\'e duality. On the other hand, this form is nonzero by our choice
of~$f$, and hence nonharmonic. This immediately implies that
$d d^*((f-\tau^{*}f)\omega)\ne0$, as claimed.

(ii) Let $\widetilde\tau$ denote some lift of~$\tau$ to the hyperbolic plane~$H_2$,
and let $\widetilde G$, $\widetilde\Gamma_i$ denote the groups of all lifts of elements
of $G$, resp.~$\Gamma_i$, to~$H_2$. Let $N(\widetilde\Gamma_1)$ denote the
normalizer of~$\widetilde\Gamma_1$ within $\operatorname{Isom}(H_2)$.
Then $\widetilde\tau\notin\widetilde G\widetilde\varphi$ for any $\widetilde\varphi\in
N(\widetilde\Gamma_1)$
because, otherwise,
the relation $\widetilde\tau\widetilde\Gamma_1\widetilde\tau^{-1}=\widetilde\Gamma_2$ would
imply
that $\Gamma_1$ and $\Gamma_2$ were conjugate in~$G$.
Note that $N(\widetilde\Gamma_1)$ consists precisely of the lifts of isometries
of~$M_1$.
Therefore the fact that $\widetilde\tau\notin\widetilde G\widetilde\varphi$
for all $\widetilde\varphi\in N(\widetilde\Gamma_1)$
implies that it is possible to choose
the $G$-invariant function~$f$
subject to the slightly stronger property
that the functions $f_1$ and $f_1^\tau$
which are induced by $f$ and $\tau^*f$ on~$M_1$, respectively,
do not differ by any isometry of~$M_1$.
Then, for any isometry $\varphi$ of~$M_1$ we can apply the
argument of~(i) to the the lift of $f_1-\varphi^* f_1^\tau$ to $M$
and conclude that now the curvature forms associated with
$\nabla_1$ and $\sigma^*\nabla_2$ are not related by
pullback by any isometry of~$M_1$.
\end{remark}

\end{document}